\documentclass[10pt]{amsart}
\usepackage{amssymb, latexsym}
\usepackage{stmaryrd}
\usepackage{graphicx}
\usepackage[scriptsize, up]{caption}
\usepackage{pdfsync, wrapfig}
\usepackage{xcolor}
\definecolor{dnrbl}{rgb}{0,0,0.5}
\definecolor{dnrgr}{rgb}{0,0.5,0}
\definecolor{dnrre}{rgb}{0.5,0,0}
\usepackage[colorlinks=true, citecolor=dnrgr, linkcolor=dnrre, urlcolor=dnrbl] {hyperref}
\usepackage{pgf,tikz}

\theoremstyle{plain}
\newtheorem{thm}{Theorem}[section]
\newtheorem{prop}[thm]{Proposition}
\newtheorem{lem}[thm]{Lemma}
\newtheorem{coro}[thm]{Corollary}

\numberwithin{equation}{section}

\newcommand{\Nat}{\mathbb{N}}

\newcommand{\restr}{\upharpoonright}  

\newcommand{\de}{\downarrow} 

\newcommand{\ml}{Martin-L\"{o}f }
\newcommand{\pz}{$\Pi^0_1$\ }

\begin{document}
\title[Kolmogorov complexity and computably enumerable sets]{Kolmogorov complexity and computably enumerable sets}

\author{George Barmpalias}
\address{{\bf George Barmpalias:} 
State Key Laboratory of Computer Science, 
Institute of Software,
Chinese Academy of Sciences,
Beijing 100190,
P.O. Box 8718,
China.}
\email{barmpalias@gmail.com}
\urladdr{\href{http://www.barmpalias.net}{http://www.barmpalias.net}}

\author{Angsheng Li}
\address{{\bf Angsheng Li:} 
State Key Laboratory of Computer Science, 
Institute of Software,
Chinese Academy of Sciences,
Beijing 100190,
P.O. Box 8718,
China.}
\email{angsheng@ios.ac.cn}

\date{06-02-2013}

\thanks{
This research was partially 
done during 
the programme `Semantics and Syntax' in the
Isaac Newton 
Institute for the Mathematical Sciences, 
Cambridge U.K.  
Barmpalias
 was supported by the
{\em Research fund for international young scientists}
number 611501-10168 from the National Natural Science Foundation of China, and
an {\em International Young Scientist Fellowship} 
number 2010-Y2GB03 from the Chinese Academy of 
Sciences.
 Li was supported by the {\em hundred talent program} of the Chinese
Academy of Sciences. Both authors received partial support by
the {\em Grand project: Network Algorithms and Digital Information} 
(number ISCAS2010-01) of the
Institute of Software, Chinese Academy of Sciences.}

\begin{abstract} 
We study the computably enumerable sets 
in terms of the:\\
(a) Kolmogorov complexity of their initial segments;\\ 
(b) Kolmogorov complexity of finite programs
when they are used as oracles.\\
We present an extended discussion of the existing research on this topic,
along with  recent developments and open problems.
Besides this survey, our main original result is the 
following characterization of the computably enumerable sets
with trivial initial segment prefix-free complexity. A computably
enumerable set $A$ is $K$-trivial if and only if the family of sets with complexity
bounded by the complexity of $A$ is uniformly computable from the halting problem.
\end{abstract}
\maketitle
\section{Introduction}
The study of computably enumerable sets 
is a major part of classical computability theory.
The main focus in Kolmogorov complexity on the other hand
is arguably strings and sequences of high complexity, hence
not (segments of) the characteristic sequences of computably enumerable sets. 
Despite this, the study of the initial segment
 Kolmogorov complexity of computably enumerable (c.e.\ for short) sets
 dates back to the work of Barzdins \cite{BarzdinsCe} and hence is nearly 
 as old as the theory of Kolmogorov complexity itself.
Moreover, as we argue in the following, it is motivated by natural questions
about c.e.\ sets and has interesting nontrivial interactions with the traditional
study of c.e.\ sets from computability theory.

This paper is concerned with
the study of the computably enumerable sets in terms of: 
\begin{itemize}
\item[(a)] the complexity of their initial segments;
\item[(b)] the complexity of finite programs when they are used as oracles.
\end{itemize}
Kolmogorov complexity is a well known measure of complexity
of strings that is based on the intuitive idea
that complicated strings do not have short descriptions.
We will focus on two variants, the plain and the prefix-free complexity,
which are formally defined in Section \ref{subse:kolcomran}.
However our discussions, as well as some of the arguments we present
(in particular the arguments of Section \ref{se:splits}), are often applicable to
different variations like monotone or process complexity 
(see \cite[Section 3.15]{rodenisbook}
for an introduction). We will not be talking about the complexity of
left or right c.e.\ reals, namely the binary expansions of reals in $(0,1)$
whose left or right Dedekind cut is computably enumerable.
This is already a well developed area (see \cite[Chapter 5]{rodenisbook}
for an elaborate presentation) and is not directly relevant to our topic.
Moreover we do not discuss the interesting topic of 
the resource bounded versions of clauses (a), (b) 
(see \cite[Chapter 7]{MR1438307} for a general introduction on resource bounded
Kolmogorov complexity and 
\cite[Theorem 7.1.3]{MR1438307} which refers to the resource bounded 
initial segment complexity of a c.e.\ set).

We start with a brief overview of the basics of Kolmogorov complexity in 
Section \ref{subse:kolcomran}.
In Section \ref{se:mecomceset} we motivate the topic of this paper with various
intuitive questions and a survey of the relevant work in the literature.
More specifically, in Section \ref{se:iscces} we study the c.e.\ sets according to 
clause (a) and we discuss to what extend the motivating questions
have been addressed in the literature. Furthermore, we present the main original 
result in this paper (whose proof is given in 
Section \ref{se:profthunifiscom}), as an answer to one of these questions.
Section \ref{subse:relcompo}
consists of an analogous discussion of 
the study of c.e.\ sets according to clause (b).
Finally Section \ref{se:cocesplse} is devoted to the special topic of c.e.\ splittings
with respect to (a) and (b). Throughout the paper we point to several open questions
in the context of each discussion. A more comprehensive and general (not
restricted to the case of c.e.\ sets) survey of the topic
of Kolmogorov complexity and reducibilities can be found in \cite{BSLsurvK}.

\section{Background on Kolmogorov complexity and randomness}\label{subse:kolcomran}
A standard measure of the complexity of a finite string 
was introduced by
Kolmogorov in \cite{MR0184801}.
The basic idea behind this approach 
is that simple strings have short descriptions relative
to their length
while complex or random strings are hard to describe
concisely.
Kolmogorov formalized this idea using the theory of computation.
In this context, Turing machines play the role of our idealized computing devices, and
we assume that there are Turing machines capable of simulating any mechanical
process which proceeds in a precisely defined and algorithmic manner.
Programs can be identified with binary strings. 

\subsection{The definition of complexity and randomness}
A string $\tau$ is said to be a
description of a string $\sigma$ with respect to a Turing machine
$M$ if this machine halts when given program $\tau$ and
outputs $\sigma$. Then the Kolmogorov complexity of $\sigma$ with respect to $M$ 
(denoted by $C_M(\sigma)$) is the length
of its shortest description with respect to $M$. 
It can be shown that there exists an \emph{optimal} machine $V$, i.e.\
a machine which gives optimal complexity for all strings, up to a certain constant number of bits.
This means that for each Turing machine $M$ there exists a constant $c$
such that $C_V(\sigma)< C_M(\sigma)+c$ for all finite strings $\sigma$.
Hence the choice of the underlying optimal machine does
not change the complexity distribution significantly
and the theory of Kolmogorov complexity can be developed without loss of generality,
based on a fixed underlying optimal machine $U$. We let $C$ denote the
Kolmogorov complexity with respect to a fixed optimal machine.

When we come to consider randomness for infinite strings, it becomes important to
consider machines whose domain satisfies a certain condition;
the machine $M$ is called \emph{prefix-free} if
it has prefix-free domain (which means that no program for which 
the machine halts and gives output is an initial segment of another).
Similarly to the case of ordinary Turing machines, 
there exists an \emph{optimal} prefix-free machine $U$ so that 
for each prefix-free machine $M$ the complexity of any string with respect to $U$
is up to a  constant number of bits larger than the complexity of it with respect to $M$.
We let $K$ denote the prefix-free complexity with respect to a fixed optimal prefix-free
machine.

In order to define randomness for infinite sequences, we consider the complexity of all
finite initial segments. A finite string $\sigma$ is said to be  $c$-{\em incompressible} 
if $K(\sigma)\geq |\sigma|-c$, where $K=K_U$.
Levin \cite{MR0366096}
and Chaitin \cite{MR0411829}  defined an infinite binary sequence  $X$ to be random 
(also called 1-random)
if there exists some constant $c$ such that all of its initial segments are $c$-incompressible.
By identifying subsets of $\Nat$ with their characteristic sequence we can also talk about
randomness of sets of numbers.
Moreover the above definitions and facts relativize to an arbitrary oracle $X$ when
the machines that we use have access to this external source of information. For example,
in this case we write $K^X$ for the corresponding function of prefix-free complexity.

This definition of randomness of infinite sequences is independent of the choice of
underlying optimal prefix-free machine, and coincides with other definitions of randomness like
 the definition given by \ml in \cite{MR0223179}. 
The coincidence of the randomness notions resulting from various different 
approaches may be seen as evidence of a robust and natural theory.

\subsection{Sets of descriptions and construction of machines}
The {\em weight} of a prefix-free set $S$ of strings, denoted $\mathtt{wgt}(S)$, is defined
to be the sum $\sum_{\sigma\in S}  2^{-|\sigma|}$. The
{\em weight} of a (oracle) prefix-free machine $M^X$ is 
defined to be the weight of its domain and is denoted $\mathtt{wgt}(M^X)$.

Prefix-free machines are most often built in terms of {\em request sets}.
A request set $L$ is a set of pairs $\langle \rho, \ell\rangle$ where $\rho$ is a string
and $\ell$ is a positive integer. A `request' $\langle \rho, \ell\rangle$
represents the intention of describing $\rho$
with a string of length $\ell$. We
define the {\em weight of the request} $\langle \rho, \ell\rangle$ to be $2^{-\ell}$.
We say that $L$ is a {\em bounded request set}
if  the sum of the weights of the requests in $L$ is less than 1.
This sum is the {\em weight of the request set $L$} and is denoted by $\mathtt{wgt}(L)$.

The Kraft-Chaitin theorem
(see e.g.\ \cite[Section 2.6]{rodenisbook}) says that for every bounded request set $L$
which is c.e., there exists a prefix-free machine $M$ such that for each $\langle \rho, \ell\rangle\in L$
there exists a string $\tau$ of length $\ell$ such that $M(\tau)=\rho$.
The same holds when $L$ is c.e.\ relative to an oracle $X$, giving a machine $M^X$.
In Section  \ref{se:profthunifiscom} 
and the proof of Proposition \ref{prop:derakini} 
we freely use this method of construction without explicit reference to the
Kraft-Chaitin theorem.

\section{Measuring the complexity of a computably enumerable set}\label{se:mecomceset}
\subsection{Initial segment complexity of computably enumerable sets}\label{se:iscces}
Computable sets have trivial Kolmogorov complexity.
In order to produce the first $n$ bits of a computable set it suffices to
have a description of $n$, since all other information can be coded in a fixed program.\footnote{Note
that we may consider $C(n)$ either by identifying it with $C(0^{n})$ or
by assuming that the underlying optimal machine prints numbers, as well
as strings. Similarly for $K(n)$.}
Computably enumerable sets may not be computable, but it is not hard to see
that the information they may absorb
in their initial segments is quite limited.
\begin{equation}\label{eq:splbarzco}
\parbox{11cm}{{\small How complex can the segments of the characteristic sequence
of a c.e.\ set be?}}
\end{equation}
Barzdins \cite{BarzdinsCe} observed that $2\log n$ is an upper bound
(up to an additive constant) of the plain Kolmogorov complexity of the first $n$ bits of 
any computably enumerable set.
Moreover he constructed a c.e.\ set $A$ such that
$C(A\restr_n)\geq \log n$ for all $n$
while Chaitin \cite{Chaitin:76} (utilizing
a result of Meyer, see \cite[Exercise 2.3.4]{MR1438307})
showed that if $\forall n\ (C(X\restr_n)\leq \log n+c)$ for some constant $c$
then $X$ is computable. On the other hand Solovay
\cite{Solovay:75} observed that there is no c.e.\ set such that 
$2\log n$ is a lower bound of the initial segment complexity of it
(even up to an additive constant).
Note that $\log n$ is an upper bound of $C(n)$. 
H\"{o}lzl, Kr\"{a}ling and Merkle 
\cite{holzlMerkKra09} observed that, in fact, for every c.e.\ set $A$
there are infinitely many $n$ such that $C(A\restr_n)$ is bounded
by $C(n)$ (plus an additive constant); this also holds for prefix-free complexity.
Finally Kummer \cite{DBLP:journals/siamcomp/Kummer96}
showed that there are c.e.\ sets $A$ such that
$C(A\restr_n)\geq 2\log n-c$ for some constant $c$ and
infinitely many $n$. Moreover he showed that the Turing degrees that
contain such complex c.e.\ sets are exactly the array non-computable
c.e.\ degrees (a well studied class of degrees from computability theory).
For a more detailed overview of these results and their proofs we refer to
\cite[Section 16.1]{rodenisbook}.


More recently, in \cite{cek2012}, it was shown that the c.e.\ sets
$X$ such that for all $n$, $C(X\restr_n)\geq \log n$  have the property that
their initial segment complexity dominates the
initial segment complexity of any other c.e.\ set (modulo an additive constant).
This result holds for both the plain and the prefix-free complexity.
Furthermore, for the case of the plain complexity, the two properties are equivalent.
Finally, a third clause in this equivalence is the linear-completeness of $X$, 
namely the property that every c.e.\ set is computable from $X$ via a use function that is bounded
by a linear function.

Our next question concerns the relation between
the overall information that is coded into a c.e.\ set and the
way that this information affects the Kolmogorov complexity
of its initial segments. Note that $C(n)$ and $K(n)$ are lower bounds for the
plain and prefix-free initial segment complexity of any sequence.
Hence we may say that the plain or prefix-free
initial segment complexity of a set is {\em trivial} if it is bounded by one of these
$C(n)$ or $K(n)$ respectively (up to an additive constant).
Such sets are also known as $C$-trivial or $K$-trivial respectively.
\begin{equation}\label{eq:spliti2nco}
\parbox{10cm}{How much information can be coded into a c.e.\ set with
trivial or `low' initial segment complexity?}
\end{equation}
As we already discussed, Chaitin \cite{Chaitin:76} showed that
sequences with trivial plain initial segment complexity are computable.
The case of prefix-free complexity turned out to be more interesting.
In \cite{MRtrivrealsH} it was shown that sequences with trivial prefix-free
initial segment complexity cannot compute the halting problem.
Hirschfeldt and Nies extended this result in \cite{MR2166184}
and showed that the amount of information that can be coded into $K$-trivial
sequences is in fact quite limited. 
On the other hand, a number of results say that there are Turing complete sets
of `very low' initial segment complexity. For example, given a 
nondecreasing unbounded $\Delta^0_2$ function $g$ there exists a 
complete c.e.\ set $A$ and a constant $c$ such that $K(A\restr_n)\leq K(n)+g(n)+c$ for all $n$.
This was demonstrated by Frank Stephan, see \cite[Section 5]{BV2010}.
In \cite{Baarsbarmp, BV2010} it was shown that this result is optimal, in the sense
that it is no longer true if one of the conditions on $g$ is removed. 
A stronger result was obtained in \cite{Baunilow}. It was shown that
there are complete c.e.\ sets of arbitrarily low complexity, with respect to the nontrivial complexities
of the c.e.\ sets. A more precise statement of this result is given in
Section \ref{subse:merelcocoen} where a formal way for comparing the initial segment complexities
of two sets is discussed.

Another topic of interest concerns the study of the ways in which sufficiently random oracles
are no better than computable oracles for performing certain computational tasks.
An early observation from \cite{deleeuw1955} is that if a set is c.e.\ relative to a 
sufficiently random oracle then it is c.e.\ without the use of an oracle.
Moreover it is well known (e.g.\ see \cite[Section 3]{Barmpalias.ea:08}) 
that if a sequence $X$ is random, then it 
is random relative to every sufficiently random sequence.
We give an example of such a result in the context of this paper.
Recall that $K(n)$ is the trivial complexity and a set is $K$-trivial if
its complexity is trivial (modulo an additive constant).
The following observation says that if the initial segment complexity of a c.e.\ set
relative to a sufficiently random oracle is trivial, then the set is already $K$-trivial.
The level of randomness that is required for this result is weak 2-randomness.
An oracle is weakly 2-random if it is not a member of any null $\Pi^0_2$ class.
\begin{prop}\label{prop:derakini}
Suppose that $X$ is weakly 2-random and $A$ is a c.e.\ set.
If $\exists c \forall n\ K^X(A\restr_n)\leq K(n)+c$ then $A$ is $K$-trivial.
\end{prop}
\begin{proof}
Given a c.e.\ set $A$ and a constant $c$, the class of oracles $X$ such that
$\forall n\ K^X(A\restr_n)\leq K(n)+c$ is a $\Pi^0_2$ class.
Hence it suffices to show that if this class is not null then $A$ is $K$-trivial.
On this assumption, by Kolmogorov's 0-1 law 
there exists a constant $d$ such that the measure of the oracles
$X$ such that $\forall n\ K^X(A\restr_n)\leq K(n)+d$ is larger than 1/2.
Let $U$ be the underlying optimal oracle machine
such that $\mu(U^X)<1/2$ for all oracles $X$.
We also assume that any computations of $U$ at stage $s$ use less than $s$ bits of the oracle.
Without loss of generality we may assume that for all $X$, if there is a $U^X$ description of length $n$
that describes some
string $\tau$ then for each $i>n$ there exists a $U^X$ description of $\tau$ of length $i$.
Given a computable enumeration $A[s]$ of $A$ we construct a prefix-free machine $M$
as follows. At stage $s+1$ let $n$ be the least number $\leq s$ such that

\begin{itemize}
\item $K_M(A\restr_n)[s]> K(n)[s]+d$ 
\item $K^{X}(A\restr_n)[s]\leq K(n)[s]+d$ for a set of oracles $X$
of measure $>1/2$
\end{itemize}
 (if there is no such $n$, do nothing).
Then enumerate an $M$-description of $A\restr_n[s]$ of length 
$K(n)[s]+d$.

It remains to prove that the request set for $M$ is bounded.
Let $\sigma_i$ be the $i$th description enumerated in $M$.
It suffices to show that $\sum_{i\leq n} 2^{-|\sigma_i|}<1$ for all $n$.
For a contradiction, suppose that $\sum_{i\leq n_0} 2^{-|\sigma_i|}\geq 1$ for some $n_0$,
and let $s_0$ be the stage where $\sigma_{n_0}$ was enumerated into our machine $M$.
Each string $\sigma_i$, $i\leq n_0$ contributes at least $2^{-|\sigma_i|-1}$ to the expected weight
of the machine $U^{X\restr_{s_0}}$, where $X\restr_{s_0}$ is any oracle of length $s_0$.
Since $\mu(U^X)<1/2$ for all oracles $X$, this expected value is also less than 1/2.
This gives the required contradiction.
\end{proof}
\noindent
Note that Proposition \ref{prop:derakini} is no longer true if we replace `weakly 2-random'
with `1-random' since there are 1-random sequences which compute all c.e.\ sets. 
Moreover the same proof shows the result in the more general case when $A$ is $\Delta^0_2$.

We would like to make a note of another way to 
study the complexity of c.e.\ sets which was introduced by
Chaitin \cite{ChaitinEntropy} and will not be studied in this paper.
The (algorithmic) entropy of a c.e.\ set $A$ was defined as the probability that
a universal c.e.\ operator enumerates $A$.
Most of the research around this concept has to do with the relationship
to another measure of complexity, namely the length of
the shortest prefix-free description of a c.e.\ index of $A$.
Solovay \cite{Solovay:75, SolovProcEntro} obtained an upper bound of the latter in terms of
the entropy function and Vereshchagin 
\cite{ipl/Vereshchagin07} improved it for the special case of finite sets.
For a more elaborate overview of this research on the algorithmic entropy of c.e.\ 
sets we refer to \cite[Section 16.2]{rodenisbook}.

\subsection{Measures of relative complexity on computably enumerable sets}\label{subse:merelcocoen}
A fruitful way to study the complexity of a sequence is
to compare it with the complexities of other sequences.
Measures of relative complexity provide a formal way to do this.
In the case of initial segment complexity, a number of such measures were
introduced by Downey, Hirschfeld and LaForte in \cite{MR2030512}.
One such measure is the $\leq_K$ reducibility defined as
\[
X\leq_K Y\iff \exists c\forall n \ (K(X\restr_n)\leq K(Y\restr_n)+c)
\]
as well as its plain complexity version $\leq_C$ which is defined similarly.
We may express the fact that $X\leq_K Y$ simply by saying that
the initial segment complexity of $X$ is less than (or equal to) the complexity of $Y$.
The result from \cite{Baunilow} that we discussed in relation to
question \eqref{eq:spliti2nco} may be formally stated as follows.
Given any c.e.\ set $B$ which is not $K$-trivial, there exists a Turing
complete c.e.\ set $A$ such that $A<_K B$.

A central topic in the study of c.e.\ sets in computability theory are the
`c.e.\ splittings', see \cite{Downey.Stob:93}. 
We say that a pair $B,C$ of c.e.\ sets is a c.e.\ splitting of $A$ if
$B\cap C=\emptyset$ and $B\cup C=A$.
One of the simplest questions that we can ask about the initial segment complexity
of c.e.\ splittings is the following. 
\begin{equation}\label{eq:splitinco}
\parbox{10cm}{Given a c.e.\ set can we split it into two c.e.\ sets of strictly less
initial segment complexity?}
\end{equation}
A positive answer was given in \cite[Section 5]{Barstris}
for both the plain and the prefix-free complexity, provided that the given set has non-trivial complexity.
In Section \ref{se:cocesplse} we give an extension of this splitting theorem, showing that the splitting may avoid
bounding the complexity of any given nontrivial $\Delta^0_2$ set.
Note that such results can be viewed as analogues of the classic Sacks splitting theorem that was proved
in \cite{Sacks:63*3} for the Turing degrees.
\begin{equation}\label{eq:spli6tinco}
\parbox{10cm}{Given a c.e.\ set can we split it into two c.e.\ sets of the same
initial segment complexity?}
\end{equation}
This question is formally addressed in Section \ref{se:cocesplse}.
In \cite{cek2012} it was given a positive answer for prefix-free complexity and a negative
answer for plain complexity.
We note that the sets that satisfy the analogue of 
(\ref{eq:spli6tinco}) in the Turing degrees are called mitotic and have been
studied extensively in computability theory, see \cite{Ladner:73, Ladner:73*1, Downey.Slaman:89}.
The existence of non-mitotic sets was shown by Lachlan in \cite{Lachlan:67}.

\begin{equation}\label{eq:spliti4nco}
\parbox{10cm}{Is there a c.e.\ set whose 
initial segment complexity is maximal amongst the c.e.\ sets?}
\end{equation}
This question has been answered negatively 
in \cite{DBLP:conf/cie/Barmpalias05} with respect to a stronger measure
$\leq_{\textrm{cl}}$,  where $X\leq_{\textrm{cl}} Y$ if $X\restr_n$
can be computed from $Y\restr_{n+c}$ for some constant $c$ and all $n$.
In \cite{Merklebarmp} an easier proof of this result was given.
By \cite{MR2030512} the partial order 
$\leq_{\textrm{cl}}$ coincides with the Solovay reducibility on the c.e.\ sets,
which is a standard measure of relative randomness for the larger class of c.e.\ reals.
In particular, it does express in some (crude) sense the relative complexity of sequences.
Other aspects of $\leq_{\textrm{cl}}$ on the c.e.\ sets were studied
in \cite{BLibT, apal/Day10}. For the case of 
$\leq_K$ and $\leq_C$ a positive answer was given recently in
\cite{cek2012}, as we discussed in the previous section.
\begin{equation}\label{eq:spliti3nco}
\parbox{10cm}{How `large' is the class of sets with initial
segment complexity bounded by the complexity of a c.e.\ set?}
\end{equation}
There are many ways to measure the largeness of a class of reals,
including determining the cardinality of the class. In the case of $\leq_C, \leq_K$
it turns out that the lower cones below a c.e.\ sets are always subsets of
$\Delta^0_2$, hence countable (e.g.\ see \cite[Section 2]{BV2010}).
It is interesting to examine if these are uniform subclasses of $\Delta^0_2$,
in the sense that they can be indexed by a single $\emptyset'$-computable
predicate. In Section \ref{se:profthunifiscom} we prove the following characterization.
\begin{thm}\label{th:unifiscom}
The following are equivalent for a computably enumerable set $A$.
\begin{itemize}
\item[(a)] $A$ is $K$-trivial;
\item[(b)] every set $X\leq_K A$ is truth-table reducible to $\emptyset'$;
\item[(c)] $\{X\ |\ X\leq_K A\}$ is uniformly computable in $\emptyset'$;
\end{itemize}
\end{thm}
This result provides an answer to question (\ref{eq:spliti3nco}) since
a uniform subclass of $\Delta^0_2$ may be considered `effectively small' while
classes that do not admit a uniform parameterization are, in a sense, `effectively large'.
A more precise treatment of this notion of `largeness' may be obtained via the use of
resource bounded measure, an approach that was developed by Jack Lutz in various
contexts and is based on the use of effective martingales.
In our case we are interested in the size of a set of reals as a subclass of 
$\Delta^0_2$. An example of such a study
is \cite{HirschTerwLcomp} where it is shown that the $\Delta^0_2$ measure of
the class of oracles that are computable by an incomplete $\Delta^0_2$ set is 0.
We may ask the same question with respect to $\leq_K$, $\leq_C$ as a formal version of
question (\ref{eq:spliti3nco}). In other words, to determine the 
$\Delta^0_2$ measure of
the class of oracles with initial segment complexity bounded by the complexity of
a given c.e.\ set.

Since Section \ref{se:profthunifiscom} is entirely devoted to the proof of Theorem
\ref{th:unifiscom}, we wish to say a few more words on its relevance with
the work of other authors. Both directions of the equivalence that it asserts are
nontrivial.
Chaitin \cite{Chaitin77ibm} observed that (by a relativization of an argument from
\cite{Loveland69}) all $K$-trivial sequences are
computable from the halting problem; equivalently, they have a computable approximation.
 However this proof is not uniform, hence it does not show that
 the $K$-trivial sequences can be listed by a single machine operating
 with oracle $\emptyset'$. 
 One of the consequences of Nies \cite{MR2166184} was that the family of
 $K$-trivial sequences is indeed uniformly $\emptyset'$-computable.
 This is the only known proof of this fact and is highly nontrivial, involving
 the full power of what is now known as the {\em decanter method}.
 
We already pointed out that given any c.e.\ set $A$,
the class of sequences with initial segment prefix-free complexity that is bounded 
by the complexity of $A$ is contained in $\Delta^0_2$.
This is merely an extension of the argument for the case where $A=\emptyset$, so
again it is nonuniform. 
We wanted to know if this class can be uniformly $\emptyset'$-computable in
any cases other than the known case where $A$ is $K$-trivial.
A positive answer would have interesting consequences on the local structures
of the $K$ degrees. For example, combined with the results in \cite{BV2010}
it would establish the existence of a pair of $\Delta^0_2$ sets that form a minimal pair
in the $K$ degrees.
However, Theorem
\ref{th:unifiscom} shows that this uniformity is a special feature that characterizes
the $K$-trivial computably
enumerable sets. 
 
\begin{coro}\label{th:unigencfiscom}
The following are equivalent for any finite collection of computably enumerable sets $A_i$, $i<k$.
\begin{itemize}
\item[(a)] There exists $i<k$ such that $A_i$ is $K$-trivial;
\item[(b)] The sets $X$ such that $X\leq_K A_i$ for all $i<k$ are truth-table reducible to $\emptyset'$.
\end{itemize}
Moreover \textup{(a)}, \textup{(b)} are equivalent to the condition that 
$\{X\ |\ \forall i<k\ (X\leq_K A_i)\}$ is uniformly computable in $\emptyset'$.
\end{coro}
\noindent
This generalization of  Theorem \ref{th:unifiscom} may be obtained by an
application of a result in \cite{Baunilow}. Namely, 
it was shown that given two c.e.\ sets
 $B, C$ such that $\emptyset <_K B$ and
 $\emptyset <_K C$ there exists a c.e. set
 $A$ such that $A\leq_K B$, $A\leq_K C$ and $\emptyset <_K A$.\footnote{At this
 point we would like to draw a parallel between the study of
 the $K$ degrees of c.e.\ sets and the $K$ degrees of \ml random sets that
 was the object of study in \cite{milleryutran, milleryutran2}.
 One of the main open questions in this study was whether there is a maximal
 element in the $K$ degrees of random reals, which is an analogue of
question (\ref{eq:spliti4nco}). Moreover it was shown that is a pair
of random reals  $X,Y$ which has no upper bound with respect to $\leq_K$.
Finally it was shown that given any finite collection $X_i$, $i<k$ 
of random reals, there exists
a random real $Y$ such that $Y<_K X_i$ for all $i<k$. The analogue of this
result for the c.e.\ sets was proved in \cite{Baunilow}.
This analogy stems from the analogy between the oscillation of
the prefix-free complexity of a random sequence between $n$ and $n+K(n)$
and the oscillation of the complexity of a c.e.\ set between $K(n)$ and $4\log n$.}

\begin{equation}\label{eq:algindceo}
\parbox{11cm}{What is `algorithmical independence'
for computably enumerable sets?}
\end{equation}
The notion of algorithmic independence of random sequences is well understood
through the concept of relative randomness. In particular, two random sequences
may be regarded algorithmically independent if each of them is random relative to
the other. Various authors have attempted to extend the notion of
algorithmic independence to a wider class of sequences.
Such formalizations have been suggested by Levin 
\cite{Levin84,Levinsdf84, DBLP:conf/focs/Levin02}
through the concept of `mutual information' of sequences, and by 
Calude and Zimand  \cite{CaludeZ10}. Although the exact relationship
between these formalizations is not known, some of them are clearly too crude
for the purpose of expressing independence for pairs of computably enumerable sets.
For example the definitions in \cite{CaludeZ10} are not sensitive to 
additive logarithmic factors. In particular, since the complexity of a c.e.\ set
is at most $4\log n$, every pair of c.e.\ sets is independent according to
\cite{CaludeZ10}.
It would be interesting to address
question (\ref{eq:algindceo}) by crafting an appropriate formalization which
expresses the informal concept of independence for c.e.\ sets.
We note that the concept of minimal pairs with respect to $\leq_K$, $\leq_C$
does express some notion of independence. The existence of minimal pairs
of c.e.\ sets with respect to $\leq_C$ was shown in
\cite{MRmerstcdhdtd} and the nonexistence 
with respect to $\leq_K$ was shown in \cite{Baunilow}.

\section{Relative compression power of computably enumerable oracles}\label{subse:relcompo}
A second way to study the c.e.\ sets with respect to Kolmogorov complexity is
to examine their power when they take the place of an oracle in the 
underlying optimal universal machine. Recall that $K^X$ denotes the prefix-free complexity
relative to oracle $X$. Hirschfeldt and Nies showed in \cite{MR2166184}
that $K^X$ does not differ from $K$ more than a constant if and only if
$X\equiv_K \emptyset$. In other words, the oracles that do not improve the 
compression
of finite programs significantly are exactly the oracles with trivial initial segment
prefix-free complexity. A natural way to compare the compression power of oracles
was introduced in \cite{MR2166184} in the form of the reducibility $\leq_{LK}$.
\[
X\leq_{LK} Y\iff \exists c\forall \sigma\ (K^Y(\sigma)\leq K^X(\sigma)+c).
\]
In other words $X\leq_{LK} Y$ formalizes the notion that $Y$ can achieve an overall
compression of the strings that is at least as good as the compression achieved by $X$.
Moreover by \cite{millerdomi} it coincides with $X \leq_{LR} Y$ which denotes the relation
that every random sequence relative to $Y$ is also random relative to $X$.
The induced degree structure is known as the $LK$ degrees.

This measure of relative complexity has been studied extensively in the literature,
although most of the publications are written in terms of $\leq_{LR}$.
Moreover the structure of  the c.e.\ sets under 
$\leq_{LK}$ has also been studied. It is not hard to see that $\leq_T$ is contained in
$\leq_{LK}$ so the c.e.\ Turing degrees have more in common with the c.e.\ $LK$ degrees
than the c.e.\ $K$ degrees. In particular, the $LK$ degree of the halting set is complete.
Despite this, the two structures are not elementarily equivalent \cite{BarmpaliasCompress}.
In particular, there are no minimal pairs in the structure of  $LK$ degrees
of c.e.\ sets.
Analogues of various questions that were discussed in Section \ref{subse:merelcocoen}
have been addressed for $\leq_{LK}$. Splitting and non-splitting theorems 
(see questions (\ref{eq:splitinco}) and (\ref{eq:spli6tinco}))
where obtained in \cite{Barmpalias.ea:08, BarmpaliasM09} and are discussed in Section \ref{subese:compceosplit}.

The size of lower $\leq_{LK}$ cones (see question (\ref{eq:spliti3nco})) in terms of cardinality
was studied in \cite{Barmpalias.ea:08, Barmpalias.Stephan:08, joelowklowwea}
and was fully determined in \cite{Barmpalias:08} (for the c.e.\ case)
and in \cite{omunca} in general.
In \cite{Barmpalias:08} it was shown that for every c.e.\ set $A$ which is not $K$-trivial
the class $\{X\ |\ X\leq_{LK} A\}$ contains a perfect \pz class.
In \cite{Baarsbarmp} it was shown how this argument can be strengthened so that
the constructed \pz class does not have $K$-trivial members
(this is one of the arguments where avoiding the $K$-trivial paths in the constructed
\pz class is highly non-trivial). On the other hand in
\cite{Bacompa} it was shown that every \pz class with no $K$-trivial paths
contains a $\leq_{LK}$-antichain of size $2^{\aleph_0}$.
The combination of these results shows that the lower $\leq_{LK}$-cone below
any c.e.\ set which is not $K$-trivial is rather large, in the following sense.
\begin{coro}
If $A$ is c.e.\ and not $K$-trivial
then $\{X\ |\ X\leq_{LK} A\}$ contains a $\leq_{LK}$-antichain of size
$2^{\aleph_0}$.
\end{coro}
\noindent
In \cite{yulrchine} it was shown that  there are perfect sets
that do not contain any $\leq_{LK}$-antichains of size $2^{\aleph_0}$. In fact, quite
interestingly, it was shown that
there exists a perfect set of reals which is a chain
with respect to $\leq_{LK}$.

We have already mentioned that $\leq_{LK}$ contains $\leq_T$.
This fact allows various standard questions to be asked about the 
structure of the c.e.\ Turing degrees inside a c.e.\ $LK$ degree.
Such issues have been studied in
\cite{Barmpalias.ea:08, Barmpalias.Stephan:08}.
Ever c.e.\ $LK$ degree contains infinite chains and antichains
of Turing degrees. In fact, the following stronger result was shown, where $|_T$ denotes
Turing incomparability.
\begin{equation}\label{eq:strucinLKaT}
\parbox{11cm}{If $A$ is a noncomputable incomplete c.e.\ set then there exist
c.e.\ sets $B, C, D$ in the $LK$ degree of $A$ such that
$B<_T A<_T C$ and $D\ |_T A$.}
\end{equation}
The proof of \eqref{eq:strucinLKaT} consists of
encapsulating a number of basic c.e.\ Turing degree constructions 
(often involving the finite and infinite injury priority method) inside an $LK$ degree.

Another question about the structure of the c.e.\ $LK$ degrees is whether it is dense.
In \cite{Barmpalias.Stephan:08} it was shown that  if $A<_{LK} B$ for two c.e.\ sets whose
$LK$ degrees have $\leq_T$-comparable c.e.\ members  then there exists
a c.e.\ set $C$ such that   $A<_{LK} C<_{LK} B$. In particular, the $LK$ degrees
of c.e.\ sets are upward and downward dense. However the density of the
c.e.\ $LK$ degrees is an open question, also stated in 
\cite[Question 9.12]{MR2248590}. A relevant question is if there are
c.e.\ sets $A,B$ such that $A<_{LK} B$ and every c.e.\ set in the $LK$ degree
of $A$ is $\leq_T$-incomparable with every c.e.\ set in the $LK$ degree of $B$.
When the sets are not required to be c.e.\ this question admits a positive answer. 
This follows from the fact that each $LK$ degree is countable \cite{MR2166184}
and the fact that there are $LK$ degrees with uncountably many predecessors \cite{Barmpalias.ea:08}.

Finally, not much is known about least upper bounds in the $LK$
degrees of c.e.\ sets. For every sets $A,B$ a natural upper bound
in the Turing (and hence the $LK$) degrees is $A\oplus B$.
However although the Turing degree of $A\oplus B$ is the least upper bound
of the degrees of $A$, $B$ the same is not necessarily true for their $LK$
degrees, even when they are computably enumerable. 
This was first noticed in \cite{MR2166184} and various results since
show that in some cases the $LK$ degree of $A\oplus B$ is, in a certain sense,
very far from being the least upper bound of the degrees of $A$ and $B$
(e.g.\ \cite[Corollary 12]{Barmpalias.Stephan:08}). 
Diamondstone \cite{DiamondstoneLR}
showed that with respect to $\leq_{LK}$ every pair of
low sets has a low c.e.\ upper bound. If we consider a pair of low
c.e.\ sets $A,B$ such that $A\oplus B\equiv_T \emptyset'$
then by  \cite{DiamondstoneLR} the $LK$ degree of $A\oplus B$ is,
in a certain sense, very far from being the least upper bound of $A, B$.
We do not know of any pair of c.e.\ sets of incomparable $LK$ degrees which
have a least upper bound in the $LK$ degrees of c.e.\ sets.
We also do not know of a pair of
c.e.\ sets whose $LK$ degrees do not have a least upper bound
in the $LK$ degrees of c.e.\ sets.

\section{Proof of Theorem \ref{th:unifiscom}}\label{se:profthunifiscom}
The implication from (a) to (b) is a result from \cite{MR2166184}.
Since the sets that are truth-table reducible to $\emptyset'$ 
are uniformly $\emptyset'$-computable,
clause (b) implies (c).
For the remaining implication assume that $A$ is c.e.\ and not $K$-trivial.
Moreover let $(X_i)$ be a
 uniformly $\emptyset'$-computable family of sets. 
This means that there is a universal computable 
approximation $(X_i[s])$  such that each $X_i[s]$
converges to $X_i$ as $s\to\infty$.
It suffices to construct a computable approximation $B[s]$ converging to set $B$
such that $B\leq_K A$ and $B\neq X_i$ for all $i$.

For the satisfaction of 
$B\neq X_i$ we pick a number $n_i$ (a witness) 
and at each stage $s$ we let $B(n_i)[s]=1-X_i(n_i)[s]$.
Since $X_i(n_i)$ converges, $B(n_i)[s]$ converges 
to an appropriate value.
For $B\leq_K A$ we will construct a prefix-free machine
$M$ such that
\begin{equation}\label{eq:negreqkb}
K_M(B\restr_k)\leq K(A\restr_k)\ \textrm{\ \ \ for all $k$}
\end{equation}
where $K_M$ denotes the prefix-free complexity relative to the machine $M$.
Recall that $K$ denotes the prefix-free complexity relative to a fixed universal
prefix-free machine $U$. Without loss of generality we may assume that
$\texttt{wgt}(U)<2^{-2}$.
The enumeration of $M$ is straightforward. At each stage we look for the
least $k$ such that (\ref{eq:negreqkb})
is not satisfied and we enumerate an $M$-description of $B[s]\restr_k$ 
of length $K(A\restr_k)[s]$.
The condition $B\leq_K A$ is satisfied provided that we
manage to keep the weight of the requests that we enumerate into $M$ bounded.
This bound may be obtained via an analysis of the requests
that are enumerated in $M$ in relation to the descriptions that are produced
in $U$.

\subsection{The model}
 Each description  that is enumerated in $M$ corresponds to a unique
description in the domain of the universal machine $U$
of the same length. Indeed, when the construction
requests $M$ to describe some $B[s]\restr_x$ at stage $s+1$, 
this is in order to achieve $K_M(B\restr_x)[s]\leq K(A\restr_x)[s]$. Hence 
the new description in $M$ corresponds to the (least) shortest description
in $U$ of 
$A\restr_x[s]$. 
Since the approximation to $B$
changes in the course of the construction,
this correspondence is not one-to-one. 
 If a $U$-description $\sigma$ corresponds to $n$ distinct $M$-descriptions
we say that $\sigma$ is used $n$ times.

Let $S_0$ be the domain of $U$ 
and for each $k>0$ let $S_k$ contain the 
descriptions in the domain of $U$ which 
are used at least $k$ times. Note that $S_{i+1}\subseteq S_i$ for
each $i$. According to the 
correspondence that we defined 
between the domains of $U$, $M$ a string $\sigma$
in the domain of $U$ that is used $k$ times incurs weight $k\cdot 2^{-|\sigma|}$ 
to the domain
of $M$. Hence (\ref{eq:Mwbou}) holds.
\begin{equation}\label{eq:Mwbou}
\mathtt{wgt}(M)\leq \sum_k \mathtt{wgt}(S_k).
\end{equation}
Note that we need not explicitly have a factor $k$ in the above sum, as
a string in $S_k$ is also in $S_i$ for $i<k$, so it is counted $k$ times.
A $U$-description is called {\em active}
at stage $s$ if $U(\sigma)[s]\subseteq A[s]$.
By the direct way that $M$ is enumerated, all descriptions that enter $S_1$ at
some stage $s$ are currently active. More generally, only
currently active strings may move
from $S_k$ to $S_{k+1}$ at any given stage. 

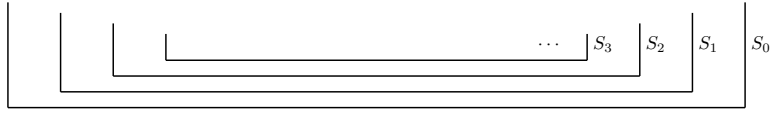
\begin{figure}
 \scalebox{0.7}{
\begin{tikzpicture}
\draw[black, thick] (-7,0)  -- (7,0);
\draw[black, thick] (-7,0)  -- (-7,2);
\draw[black, thick] (7,0)  -- (7,2);
\draw[black, thick] (-6,0.3)  -- (6,0.3);
\draw[black, thick] (-6,0.3)  -- (-6,1.8);
\draw[black, thick] (6,0.3)  -- (6,1.8);
\draw[black, thick] (-5,0.6)  -- (5,0.6);
\draw[black, thick] (-5,0.6)  -- (-5,1.6);
\draw[black, thick] (5,0.6)  -- (5,1.6);
\draw[black, thick] (-4,0.9)  -- (4,0.9);
\draw[black, thick] (-4,0.9)  -- (-4,1.4);
\draw[black, thick] (4,0.9)  -- (4,1.4);
\node  at (7.3,1.2) {$S_0$};
\node  at (6.3,1.2) {$S_1$};
\node  at (5.3,1.2) {$S_2$};
\node  at (4.3,1.2) {$S_3$};
\node  at (3.3,1.2) {$\cdots$};
\end{tikzpicture}}
\caption{{\small Infinite nested decanter model.}}
\label{fig:treq}
\end{figure}
The sets $S_k$ may be visualized as the nested containers of the
infinite decanter model of Figure \ref{fig:treq}. As the figure indicates,
descriptions may move from $S_k$ to $S_{k+1}$ but they 
also remain in $S_k$. 
If $B(n_i)[s]\neq B(n_i)[s+1]$ at some stage $s+1$ for some witness $n_i$,
some strings may move from $S_k$ to $S_{k+1}$ for various $k\in\Nat$ at some later stage.
 In this case we say that these strings
were reused by $n_i$. In order to hold a unique $n_i$ responsible for the reuse of a $U$-description,
we only consider the one with the least $i$ that reused the string. In this way,
every time that a description is reused, it is because a unique $n_i$ (the marker that
changed the least digit of $B$ that caused the enumeration of the new $M$-description).

\subsection{Movable markers and auxiliary machines}\label{subse:movmauxma}
A single witness $n_i$ may reuse each description in the
domain of $U$ at most once, thus contributing to a 2:1 
correspondence between the domains of $M$ and
$U$. However $t$ many witnesses may create a $2^t:1$ correspondence
between $U,M$ which may inflate the bound in (\ref{eq:Mwbou}).
For this reason the witnesses $n_i$ will be  movable and
obey the following rules (provided that they are defined).
\begin{itemize}
\item $n_i[s]< n_{i+1}[s]$ and $n_i[s]\leq n_{i}[s+1]$;
\item If $B(n_i)[s]\neq B(n_i)[s+1]$ then $n_{i+1}[s+1]$ moves to
a {\em large} value.
\end{itemize}
The witnesses $n_i[s]$ will ultimately reach limits
$n_i$.
In order to ensure that the descriptions in $U$ that are reused many times
have sufficiently small weight (i.e.\ they describe sufficiently complex 
strings), for each witness $n_i$ we enumerate a prefix-free machine
$N_i$ during the construction. The purpose of $N_i$ is
to achieve $\forall z\ (K_{N_i}(A\restr_z)\leq K(z)+c_i)$
for some constant $c_i$.
Since $A$ is not $K$-trivial, this will ultimately fail. However this failure
will help to demonstrate that $n_{i+1}$
converges. 
Each time $n_i$ moves, 
the value of $c_i$ increases by 1 and $N_i$ is initialised (i.e.\ all
of its computations
are deleted).
Such an event is described as an `injury' of $n_i$.
In particular, if at some stage $s$
a witness $n_i$ moves while $n_j$, $j<i$ remain constant
this causes $n_x$, $x\geq i$ to be injured,
which has the following consequences:
\begin{itemize}
\item $n_x$, $x> i$ become undefined;
\item the values $c_x, x\geq i$ increase by 1
\item $N_x$ is initialized.
\end{itemize}
Each witness will only be injured finitely many times.
We let $c_i[s]$ denote the 
value of $c_i$ at stage $s$ and
set $c_i[0]=i+4$.

At each stage $s$ let $t_i[s]$ be the least number $t$ such that
$K_{N_i}(A\restr_t)[s]> K(t)[s]+c_i[s]$.
Each witness $n_i$ has the incentive to move 
its successor $n_{i+1}$ at some stage $s+1$ if it 
observes a set of descriptions of segments of $A[s]$ 
that are longer than  $n_{i+1}$, of sufficient weight.
This weight is determined by the threshold $q_i[s]$ 
and is set to $2^{-K(t_i)[s]-c_i[s]}$. Witness $n_i$ may move
its successor $n_{i+1}$
either because the above weight exceeds the threshold or because
the approximation to $B(n_i)$ changes. 
Due to this second incentive for movement,
a new parameter $p_i[s]$ will tune the threshold to an appropriate value.
In particular, witness $n_i$ requires attention at stage $s+1$ if
it is defined and one of the following occurs:
\begin{itemize}
\item[(a)] $B(n_i)[s]=X_i(n_i)[s]$;
\item[(b)] $\sum_{n_{i+1}[s]<j\leq s} 2^{-K(A\restr_j)[s]} \geq q_i[s] - p_i[s]$.
\end{itemize}
At each stage $s+1$ the machines $N_i$ will be adjusted according to
changes of $K(n)$ for $n<t_i[s]$. This is done by running the 
following subroutine.
\begin{equation}\label{eq:tsubroureprele}
\parbox{9cm}{For each $i\leq s$ and each $n< t_i[s]$, 
if $K(n)[s+1]<K(n)[s]$ then enumerate an $N_i$-description
of $A[s]\restr_n$ of length $K(n)[s+1]+c_i[s]$.}
\end{equation}
A {\em large} number at stage $s+1$ is one that is larger than any 
number that has been
the value of any parameter in the construction up to stage $s$.

\subsection{Construction of \texorpdfstring{$B, M, N_r$}{B,M,N}}
At stage 0 place $n_0$ on $0$.
At stage $s+1$ run (\ref{eq:tsubroureprele}).
If none of the currently defined witnesses requires attention,
let $k$ be the largest number with $n_k[s]\de$, let 
$z$ be the least number that is bounded by $s$
and the current value of some marker
such that $K_M(B\restr_z)[s]>K(A\restr_z)[s]$ and 
\begin{itemize}
\item place $n_{k+1}$ on the least {\em large} number;
\item  enumerate an $M$-description of $B[s]\restr_z$ of length
$K(A\restr_z)[s]$.
\end{itemize}
Otherwise let $x$ be the least number such that $n_x$ 
requires attention and
define $n_{x+1}[s+1]$ to be a {\em large} number.
Moreover declare $n_i[s+1]$, $i> x+1$ undefined,
set $c_j[s+1]=c_j[s]+1$ for each $j>x$ and if (a) applies
set $B(n_x)[s]=1-X_i(n_x)[s]$. If (b) applies set $p_x[s+1]=0$ and
enumerate an $N_x$-description of $A\restr_{t_x}[s]$ of
length $K(t_x)[s] +c_x[s]$. 
If (b) does not apply set 
$p_x[s+1]=p_x[s]+\sum_{n_{x+1}[s]<j\leq s} 2^{-K(A\restr_x)[s]}$. 

\subsection{Verification}\label{subse:versingt}
When $n_{i+1}$ is first defined
at some stage $s$ it takes a {\em large} value so $t_i[s]< n_{i+1}[s]$.
Moreover $t_i$ can only increase (without $n_i$ being injured) when $N_i$ computations are enumerated
on strings of length $t_i$,
which happens only when
$n_{i+1}$ moves. Hence by induction we have (\ref{eq:timirele}).
\begin{equation}\label{eq:timirele}
\parbox{8cm}{For all $i,s$, if 
$n_{i+1}[s]$ is defined then $t_i[s]< n_{i+1}[s]$.}
\end{equation}
If $K(n)$ decreases at some stage $s+1$
for some $n<t_i[s]$, subroutine
(\ref{eq:tsubroureprele}) will ensure that $K_{N_i}(A\restr_n)[s+1]\leq K(n)[s+1]+c_i[s+1]$.
Hence $t_i$ may only decrease at
$s+1$ if $A[s+1]\restr_{t_i[s]}\neq A[s]\restr_{t_i[s]}$, which implies (\ref{eq:tmovmonle}).
\begin{equation}\label{eq:tmovmonle}
\parbox{10cm}{If  $A[s]\restr_{t_i[s]}=A[s+1]\restr_{t_i[s]}$ then $t_i[s]\leq t_i[s+1]$.}
\end{equation}
The enumeration of descriptions into $N_i$ occurs with overhead $c_i$,
in the sense that at stage $s$ any description of a string of length $n$
that is defined in $N_i$ has length $K(n)[s]+c_i[s]$. This, along with the initialisation
of $N_i$ upon an injury of $n_i$  implies (\ref{eq:Nidescovb}).
\begin{equation}\label{eq:Nidescovb}
\parbox{8.5cm}{At any stage $s$ the weight of the $N_i$-descriptions that describe initial segments of 
$A[s]$ is less than $2^{-c_i[s]}$.}
\end{equation}
For each $i$ there is a machine $N_i$ as
prescribed in the construction.
\begin{lem}\label{le:nibound}
For each $i$ the weight of the requests in $N_i$ is bounded.
\end{lem}
\begin{proof}
The weight of the requests that are enumerated in $N_i$
by subroutine (\ref{eq:tsubroureprele}) is bounded by
the weight of the domain of $U$, which is at most $2^{-2}$.
In order to calculate the weight of the requests
that are enumerated by the main construction, 
let $s_j$ be the stages where requests are enumerated into $N_i$
(within an interval of stages where $N_i$ is not initialised).
Note that during each interval $[s_j, s_{j+1})$ the 
successor witness $n_{i+1}$ may move
many times, thereby increasing $p_i$ which becomes 0 at $s_{j+1}$.
At each $s_j$  
the witness $n_{i+1}$ moves to a {\em large} value  and
the weight of the request that is issued in $N_i$ is
$q_i[s]\leq \sum_{x\in (n_{i+1}[s_{j-1}], s]} w_x$,
where $s_{-1}=0$ and $w_x$ is the weight of the descriptions in $U$ that
describe strings of length $x$.
Hence by induction the weight of the requests that are enumerated in $N_i$
in this way is also bounded by the weight of the domain of $U$.
Hence $\mathtt{wgt}(N_i)\leq 2^{-2}+2^{-2}=2^{-1}$.
\end{proof}

In order to calculate a suitable upper bound for each
$\mathtt{wgt}(S_k)$ of (\ref{eq:Mwbou}) we need (\ref{eq:bfolocoAnsk}).
Recall that  if an
$M$-description $\tau$ of  $B[s]\restr_z$ 
of length $K(A\restr_z)[s]$ is enumerated at stage $s+1$ and
$\sigma$ is the leftmost string of length  $K(A\restr_z)[s]$ 
describing $A[s]\restr_{z}$ via $U[s]$,
then we say that $\tau$ corresponds to $\sigma$.
Suppose that $\tau$ already corresponds to $\sigma$ at some stage $s$,
and at a latter stage $t$ some new $\tau'$ gets corresponded to $\sigma$.
Then the segments that are described by $\tau, \tau'$ (which are approximations of $B$)
must differ at some least position $k$. This change in the approximation to $B$
must have occurred due to the change of $B(k)$ by some marker $n_i$
that rested on $k$ at some previous stage. In this case we say that 
$\sigma$ is reused by $n_i$.
In this way, every reused $U$-description 
 is reused by a unique marker $n_i$. 
 Also recall that a $U$-description $\sigma$ is active at stage $s$ if
 $U(\sigma)$ is an initial segment of $A[s]$.
\begin{equation}\label{eq:bfolocoAnsk}
\parbox{11cm}{Suppose that during the interval of stages $[k,r]$ 
the witness $n_j$  is not injured. 
Then the weight of the strings that is reused by $n_{j+1}$ during
this interval which remain active at stage $r$ is at most $2^{-c_j[k]}$.}
\end{equation}
The witness $n_{j+1}$ can only be held responsible for
the reuse of $U$-descriptions at stages where it is defined
and stable (i.e.\ it does not move). Only at those stages it can
change $B(n_{j+1})$, thus causing a reuse of $U$-descriptions. When $n_{j+1}$ moves at
some stage $s+1$,
the weight of the $U$-descriptions it is held responsible for reusing
since its last move is the increase in $p_j$ since the last stage it moved
(excluding the stage where it last moved).
In the special case where $p_j[s+1]$ is set to 0, the  
weight of the $U$-descriptions it is held responsible for reusing
since the last time that $p_j$ was set to 0 (excluding this stage) is $q_j[s]$ (which is essentially
the addition of all the increases in $p_j$ since that stage).

Let $(s_i)$ be a monotone enumeration of those stages $s+1$ where $p_j$ is set to 0.
According to \eqref{eq:timirele}, the weight
we wish to bound for \eqref{eq:bfolocoAnsk} consists only of the increases
obtained in the intervals $[s_i, s_{i+1}]$ such that
$A\restr_{t_i[s_i]}$ is a prefix of $A[r]$. 
Let $(s^{\ast}_i)$ be the subsequence of $(s_i)$ consisting of those stages.
Then the weight we wish to bound for \eqref{eq:bfolocoAnsk}
is bounded by $\sum_i q_j[s_i^{\ast}]$.
According to (\ref{eq:tmovmonle}) we have that
$t_j[s^{\ast}_i]\leq t_j [s^{\ast}_{i+1}]$ and by the choice of
$(s^{\ast}_i)$ we have $t_j[s^{\ast}_i] < t_j [s^{\ast}_{i+1}]$.
Now recall that $q_j[s_i^{\ast}]=2^{-K(t_j)[s_i^{\ast}]-c_j[k]}$.
So $\sum_i q_j[s_i^{\ast}]\leq \sum_i 2^{-c_j[k]-K(i)}$ which is bounded by
$2^{-c_j[t]}$.
 This concludes the proof of 
(\ref{eq:bfolocoAnsk}).

\begin{lem}\label{le:Mbound}
The weight of the requests that are enumerated in $M$ is finite.
\end{lem}
\begin{proof}
Since only strings in the domain of $U$ are used,
$\mathtt{wgt}(S_1)<2^{-2}$ and since $S_2\subseteq S_1$
we also have $\mathtt{wgt}(S_2)<2^{-2}$.  Let $k>1$.
Every entry of a string into $S_{k+1}$ 
is due to some witness $n_x$ which reused it when it was already in
$S_k$. Since $k>1$, this string entered $S_k$ due to another witness 
$n_y$ with $y>x\geq 0$ which subsequently moved to a {\em large}
value.
Inductively, that string entered $S_1$ due to a witness $n_{z+1}$ 
with $z\geq k-2$. 
Fix $z$, let $S_k^z$ contain the strings in $S_k$ that entered $S_1$
due to witness $n_{z+1}$ and let $(s_i)$ be the increasing sequence of 
stages where witness $n_{z}$  (i.e.\ the predecessor of $n_{z+1}$) is injured. 
Note that at this point we do not assume that
$(s_j)$ is finite.

The strings that move from
$S^z_k$ to $S_{k+1}$  are currently active and may be
divided into the packets of strings that enter $S_1$
due to $n_{z+1}$ during each interval $[s_i, s_{i+1}]$. 
Note that we need not consider the endpoints $s_i, s_{i+1}$
as in those stages marker $n_{z+1}$ is not charged with such movement of strings
through the containers. 
According to (\ref{eq:bfolocoAnsk})
the weight of the strings in the $i$th packet  
which are still active at $s_{i+1}$ 
(hence, may move from $S_k^z$ to $S_{k+1}$ later on) 
is bounded by $2^{-c_z[s_i-1]}$.
So the weight of the strings that enter
$S_{k+1}$ from $S^z_k$ is bounded by $\sum_j 2^{-c_z[s_j-1]}$.
Since $c_z[s_{j+1}-1]=c_z[s_j] > c_z[s_{j}-1]$ for
all $j$,  this weight is bounded by $\sum_j 2^{-c_z[0]-j}=2^{-c_z[0]+1}$.
Since $c_z[0]=z+4$ this bound becomes
$2^{-z-3}$.
Since $S_k=\cup_{z\geq k-2} S_k^z$
the total weight of the strings
that enter
$S_{k+1}$ from $S_k$ is bounded by $\sum_{z\geq k-2} 2^{-z-3}=2^{-k}$. 
Therefore by (\ref{eq:Mwbou}) the weight of $M$ is finite.
\end{proof}
\noindent
By Lemma \ref{le:Mbound} the machine $M$ prescribed in the construction
exists.
The following proof uses the fact that each $N_i$ is a prefix-free machine,
which was established in Lemma \ref{le:nibound}.
\begin{lem}\label{le:miconv}
For each $i$ the witness $n_i$ moves
only finitely many times i.e.\ $n_i[s]$ reaches a limit.
\end{lem}
\begin{proof}
Assume that this holds for all $i\leq k$. Then
 $n_k$
stops moving after some stage $s_0$.
The construction will define $n_{k+1}$ at some later stage $s_1$. 
In the following we write $n_k$
for the the limit of $n_k[s]$ when $s\to\infty$.
Similarly, $c_k[s]$ reaches a limit $c_k:=c_k[s_0]$ at $s_0$.
Since $A$ is not $K$-trivial
 there is some least $j$ such that
$K_{N_k}(A\restr_j)>K(j)+c_k$. If $s_2>s_1$ is a stage where 
the approximations to $A\restr_j$ and $K(i)$, $i\leq j$ 
have settled then the approximations to $t_k$, $q_k$ also reach a limit
by this stage.

Let $s_3>s_2$ be a stage at which the approximation to
the membership of $n_k$ in $X_k$ has reached a limit.
If  $n_{k+1}$ moved after stage $s_3$ this would be solely due to
clause (b) of Section \ref{subse:movmauxma}. 
Hence at such a stage the construction would enumerate an 
$N_k$-description of $A\restr_j$
of length $K(j)+c_k$ which contradicts the choice of $j$.
Hence $n_{k+1}$ reaches a limit by stage $s_3$ and this 
concludes the induction step.
\end{proof} 
\noindent
We may now show that $B\leq_K A$.
\begin{lem}
The approximation to $B$ converges and \textup{(\ref{eq:negreqkb})} is met.
\end{lem}
\begin{proof}
If $k$ is not the limit of some witness $n_i$ then $B(k)$
will only change finitely often, since witnesses are
defined monotonically and redefined to {\em large}
values. On the other hand for the limit value $n_i$ of the $i$th witness
the construction will stop changing the approximation to $B(n_i)$ once
the approximation to $X(n_i)$ stops changing. Hence the approximation
to $B$ converges.
Moreover the constant enumeration of $M$ descriptions by the construction ensures
that (\ref{eq:negreqkb}) holds for $B$.
\end{proof}

Finally, we may
conclude that $B$ is not a member of 
the given uniformly $\emptyset'$-computable
family of sets.
Given $k$ consider the limit $n_k$ of the $k$th witness
that was established in
Lemma \ref{le:miconv}.
The construction explicitly ensures
that the final value of $B(n_k)$ is $1-X(n_k)$. Hence $B\neq X_k$ for all $k$.
This concludes the verification of the construction
and the proof of Theorem \ref{th:unifiscom}.

\section{Computably enumerable splittings and Kolmogorov complexity}\label{se:splits}\label{se:cocesplse}
A computably enumerable (c.e.) splitting of a c.e.\ set $A$ is a pair of disjoint c.e.\ sets $B,C$
such that $A=B\cup C$. This notion
has been the subject of interest for many researchers in computability.
For example it plays a special role in the
study of the lattice of the c.e.\ sets under inclusion
(see \cite{Downey.Stob:90}), which is a very developed
area of computability theory (see \cite[Chapter X]{MR882921} for an overview).
Moreover the Turing degrees of c.e.\ splittings have been studied extensively
(see \cite[Chapter XI]{MR882921} for an overview). For a comprehensive survey
of c.e.\ splittings in computability theory we suggest \cite{Downey.Stob:93}.

In this section we discuss c.e.\ splittings in the context of Kolmogorov complexity.
For example, we are interested in the initial segment complexity of the members $B,C$
of the splitting given the complexity of the original set $A$. In
Section \ref{subse:inicomcesp}
we show that some of the classical theory of c.e.\ splittings can be generalized
(both in terms of results and in terms of methods) to the context of initial segment complexity.
Section \ref{subese:compceosplit} discusses analogous results when c.e.\ sets are  used as
oracles for compressing finite programs. 
Our presentation has a bias toward the 
prefix-free version of Kolmogorov complexity, but most of the results and methods that we discuss
in this section also hold for plain Kolmogorov complexity.

The results that we present regarding the structure of the $K$ degrees 
(comparing the initial segment complexity
of c.e.\ sets) 
and the structure of the $LK$ degrees (comparing the compression strength of c.e.\ oracles)
revolve around the same structural questions and often have the same answers.
However in general the two structures are very different, even in the case of c.e.\ sets.
For example, $\leq_{LK}$ is an extension of Turing reducibility but, as many results
in \cite{MRmerstcdhdtd} demonstrate, there is no direct connection between $\leq_K$ and
Turing reducibility. In particular, there is a complete c.e.\ $LK$  degree
(i.e.\ maximum amongst the c.e.\ $LK$ degrees) but
it is not known if there is a maximum c.e.\ $K$  degree. In our view this is unlikely, and a
more interesting open question is whether there exist maximal c.e.\ $K$ degrees. 

The general programme of transferring results and methods from classical
computability theory to the study of Kolmogorov complexity (as well as its limitations)
is a fascinating topic and a critical discussion of it and its relation with
arithmetical definability may be found in \cite[Section 1]{BV2010}.

\subsection{Initial segment complexity and c.e.\ splittings}\label{subse:inicomcesp}
Given a computably enumerable set $A$, we are interested in the initial segment 
complexity of the members of the various splittings of $A$.
It is not hard to see that the Kolmogorov complexity of $A\restr_n$ is
equal to the Kolmogorov complexity of the last number $<n$ that enters $A$ in 
a computable enumeration of it.
A basic result from \cite[Section 5]{Barstris} and \cite[Chapter 2]{Tom:10}
is that the analogue of the Sacks splitting theorem (e.g.\ see \cite[Theorem 3.1]{MR882921})
holds in the context of plain or prefix-free Kolmogorov complexity.
\begin{equation}\label{eq:splitcl}
\parbox{10cm}{If $A$ is c.e.\ set and 
$A >_K \emptyset$ then $A$ is the disjoint union of two 
c.e.\ sets $A_0$, $A_1$ such that 
$A_0 |_K A_1$ and $A_0,A_1 <_K A$.}
\end{equation}
In particular, if a c.e.\ set has nontrivial initial segment
complexity then it can be split into two c.e.\ sets with strictly less
initial segment complexity. 
This fact also holds for the plain complexity
$C$ in place of $K$. Note that the assumption
$A >_K \emptyset$ is stronger than $A >_C \emptyset$
which is merely another way to say that $A$ is noncomputable. 
For various combinations of (\ref{eq:splitcl})
with other splitting theorems (like the classic Sacks splitting theorem)
we refer to \cite[Section 5]{Barstris} and \cite[Chapter 2]{Tom:10}.

We wish to combine (\ref{eq:splitcl}) with cone avoidance.
First, we demonstrate
a cone avoidance argument in the $K$ degrees in isolation, in terms of \pz classes.
A tree is a computable function from strings to strings which preserves
the compatibility and incompatibility relations. A real $X$ is a path through a tree $T$
if all of its initial segments belong to the downward closure of the image of $T$
under the prefix relation. We denote the set of infinite paths through a tree $T$ by $[T]$.
\begin{thm}[Cone avoidance for $\leq_K$ and \pz classes]\label{th:cafpik}
If $A$ is $\Delta^0_2$ and $A\not\leq_K\emptyset$ then
there exists a \pz class of reals $X$ such that $X\not\leq_K\emptyset$
and $A\not\leq_K X$.
\end{thm}
\begin{proof}
The task of avoiding $K$-trivial members when constructing a \pz class
has been extensively discussed in 
\cite{Kucera.Slaman:07, Barmpalias.Stephan:08, BV2010,  Baarsbarmp}.
In order to focus on the cone avoidance ideas we only prove a simpler
version of Theorem \ref{th:cafpik} which merely requires the \pz class to be perfect.
The combination of this construction with extra requirements guaranteeing that the
\pz class does not have $K$-trivial members is along the lines
of the argument discussed in \cite{Kucera.Slaman:07}. In particular, it is
not as simple as the case in \cite{BV2010} but it is much simpler than the argument
discussed in \cite{Barmpalias.Stephan:08} (which is in turn simpler than the one in
\cite{Baarsbarmp}).

Let $A[s]$ be a computable approximation to $A$.
We approximate a perfect \pz tree $T: 2^{<\omega}\to 2^{<\omega}$
and ensure that $A\not\leq_K X$ for all paths $X$ through $T$.
Let $T_{\sigma}$ denote the image of $\sigma$ under $T$ and let 
$T_{\sigma}[s]$ denote its approximation at (the end of) stage $s$.
We will satisfy the following requirements.
\[
R_e: \ \forall\sigma\in 2^{e}\ \exists n\leq |T_{\sigma}|\ [K(A\restr_n)>K(T_{\sigma}\restr_n)+e].
\]
We define the length of agreement $\ell_{\sigma}$ 
for each string $\sigma$ of length $e$ in order to monitor the
satisfaction of $R_e$ with respect to $T_{\sigma}$. Let $\ell_{\sigma}[s]$ be the 
largest number
$n\leq s$ such that $K(A\restr_j)[s]\leq K(X\restr_j)[s]+e$ for some extension $X$ of
$T_{\sigma}[s]$ in $[T[s]]$ and all $j<n$. We view the approximations 
to each $T_{\sigma}$ as a movable marker.
We say that $T_{\sigma}$ requires attention at stage $s+1$ if 
$\ell_{\sigma}[s]$ is larger than $|T_{\sigma}[s]|$.

At stage 0 we let $T_{\sigma}[0]=\sigma$ for all $\sigma$.
At stage $s+1$ let $\sigma$ be the least string  of length at most
$s$ that requires attention  
(if there is no such string, do nothing).
Let $T_{\sigma}[s+1]=T_{\rho}[s]$ where $\rho$ is the least extension of
$\sigma$ such that $|T_{\rho}[s]|$ is larger than $\ell_{\sigma}[s]$
and  $K(A\restr_j)[s]\leq K(T_{\rho}[s]\restr_j)[s]+e$ for all $j<|T_{\rho}[s]|$.
Moreover for each string $\eta$ 
let $T_{\sigma\ast\eta}[s+1]=T_{\rho\ast\eta}[s]$.

We start the verification of the construction by noting that the reals that
are paths through all trees $T[s]$ form a \pz class.
Moreover since the nodes $T_{\sigma}$ can only move to
existing nodes in $T[s]$ at stage $s+1$ in
a monotone fashion, it follows that $[T[s+1]]\subseteq [T[s]]$
for all stages $s$. Hence if we show that each node $T_{\sigma}$ reaches a limit
then the paths through the limit tree $T$ form a perfect \pz class.

It remains to show that $T[s]$ reaches a limit $T$ such that $R_e$ is satisfied for all $e$.
We show this by induction. Assume that by stage $s_0$ all
nodes $T_{\sigma}$ with $|\sigma|\leq e$ have reached a (finite) limit
such that $R_i$ is satisfied for each $i\leq e$.
Fix a string $\sigma$ of length $e+1$. If $T_{\sigma}$ is redefined infinitely often,
then $T_{\sigma}[s]$ converges to a computable real $X$ such that
$K(A\restr_n)\leq K(X\restr_n)+e+1$ for all $n$. This is a contradiction since $A$
is not $K$-trivial. Hence $T_{\sigma}[s]$ reaches a finite limit $T_{\sigma}$ such that
$K(A\restr_n)> K(T_{\sigma}\restr_n)+e+1$ for some $n\leq |T_{\sigma}|$.
Hence $R_{e+1}$ is satisfied. This completes 
the induction step and the proof.
\end{proof}
\noindent
The methods of construction behind (\ref{eq:splitcl})
(from \cite[Section 5]{Barstris} and \cite[Chapter 2]{Tom:10})
and Theorem \ref{th:cafpik} can be combined in a proof of the following
enhanced splitting theorem.

\begin{thm}[Splitting with cone avoidance for $\leq_K$]\label{thrm_splitk} 
Let $A$ be a c.e.\ set such that $A \not\leq_K \emptyset$
and let $X$ be a $\Delta^0_2$ set such that $X \not\leq_K \emptyset$. 
Then $A$ is the union of two disjoint c.e.\ sets $A_0, A_1$ such that 
$A_0 |_K A_1$, $A_i <_K A$ and  $X\not\leq_K A_i$ for $i=0,1$.
\end{thm}
\begin{proof}
In the course of enumerating the elements of $A$ into $A_0$ and $A_1$ we
satisfy the following requirement for $e\in\Nat$ and $i=0,1$.
\[
R_{\langle e,i \rangle}: \exists n \ \big[K(A_{1-i} \upharpoonright_n) > K(A_{i} \upharpoonright_n) + e\big].
\]
\noindent
Thus we ensure that $A_0\not\leq_K A_{1}$ and $A_1\not\leq_K A_{0}$. 
By \cite[Lemma 5.1]{Barstris} we also get $A_0,A_1 <_K A$. 
Define the  {\em length of agreement} $l(e,i)[s]$ of $R_{\langle e,i\rangle}$ at stage $s$ to
be the largest $n\leq s$ such that 
$\forall j<n\ (K(A_{1-i} \upharpoonright_j)[s] \leq K(A_i \upharpoonright_j)[s]+e)$.
We also need to meet the following requirement for all $e\in\Nat$ and $i=0,1$.
\[
N_{\langle e,i\rangle}: \ \exists n \ [K(X\restr_n)>K(A_i\restr_n)+e].
\]
We define the length of agreement $m(e, i)$ 
in order to monitor the satisfaction of $R_{\langle e,i\rangle}$. 
Let $m(e, i)[s]$ be the 
largest number $n\leq s$ with the property that for all $j<n$,   
$K(X\restr_j)[s]\leq K(A_i\restr_j)[s]+e$.
Let the restraint imposed by 
$R_{\langle e,i\rangle}, N_{\langle e,i\rangle}$ 
on the enumeration of $A_i$ 
at stage $s+1$ be given by
\[
r(e,i)[s] = \max_{t \leq s}\{l(e,i)[t], m(e,i)[t], e\}.
\]
\noindent
Note that by definition the restraint is non-decreasing in the stages $s$. 
Let $A_i[0]=\emptyset$ for $i=0,1$ and without loss of generality assume that at each stage
exactly one element is enumerated in $A$.

\ \paragraph{{\em Construction}}
If $x \in A[s+1]-A[s]$ consider the least $\langle e,i \rangle$ such that $x \leq r(e,i)[s]$ and
enumerate $x$ into $A_{1-i}$. 

\ \paragraph{{\em Verification}}
By induction we show that 
for all $\langle e,i\rangle$ requirements $R_{\langle e,i\rangle}, N_{\langle e,i\rangle}$
are met and $r(e,i)$ reaches a limit.
Suppose that there is a stage $s_0$ such that for all $\langle e',i' \rangle < \langle e,i \rangle$ 
the requirements $R_{\langle e',i' \rangle}, N_{\langle e',i' \rangle}$ 
are met  and $r(e',i')[s]$ remains constant for all $s\geq s_0$.
Without loss of generality we may assume that $s_0$ is large enough so that all numbers enumerated in $A$
after $s_0$ are larger than the final values of $r(e',i')$, $\langle e',i' \rangle < \langle e,i \rangle$.
By the choice of $s_0$, after that stage all numbers enumerated 
into $A_i$ will be larger than the current value of $r(e,i)$. 

For a contradiction, suppose that 
$R_{\langle e,i \rangle}$ is not met. Then the length of agreement 
$l(e,i)$ and the restraint $r(e,i)$ tend to infinity.
Since $r(e,i)[s]$ is nondecreasing in $s$ it follows that
$A_i$ is computable; hence $K$-trivial. 
Since $R_{\langle e,i \rangle}$ is not met, it follows that $A_{1-i}$ is $K$-trivial.
Since $A$ is the disjoint union of $A_0$ and $A_1$ we have $A\equiv_T A_0\oplus A_1$. 
Then $A$ is $K$-trivial, given that $K$-triviality is closed under the join operator.
This contradicts the assumption about $A$. Hence $R_{\langle e,i \rangle}$ is met.

For a second contradiction, assume that
$N_{\langle e,i \rangle}$ is not met. Then the length of agreement 
$m(e,i)$ and the restraint $r(e,i)$ tend to infinity.
Since $r(e,i)[s]$ is nondecreasing in $s$ it follows that
$A_i$ is computable, hence $K$-trivial. 
Since $N_{\langle e,i \rangle}$ is not met, it follows that $X$ is $K$-trivial.
This contradicts our assumption about $X$. Hence $N_{\langle e,i \rangle}$ is met.

To conclude the induction step (and the proof) it suffices to show that $r(e,i)[s]$ reaches a limit
as $s$ tends to infinity. But this is a direct consequence of its definition and the fact that
$R_{\langle e,i \rangle}, N_{\langle e,i \rangle}$ are met. 
\end{proof}
\noindent
The proof of Theorem \ref{thrm_splitk} can be written for $\leq_C$ instead 
of $\leq_K$ with no essential changes. This trivial modification gives the following
analogue.
\begin{thm}[Splitting with cone avoidance for $\leq_C$]\label{thrm_splitc} 
Let $A$ be a c.e.\ set such that $A \not\leq_C \emptyset$
and let $X$ be a $\Delta^0_2$ set such that $X \not\leq_C \emptyset$. 
Then $A$ is the union of two c.e.\ sets $A_0, A_1$ such that $A_0\cap A_1=\emptyset$,
$A_0 |_C A_1$ and $A_i <_C A$, $X\not\leq_C A_i$ for $i=0,1$.
\end{thm}

The above results establish  the existence of c.e.\ splittings of
strictly lesser complexity. We wished to know if it is always possible to split
a c.e.\ set into two parts of the same initial segment complexity.
It is not hard to see that there exist specially crafted 
c.e.\ sets that can be split into two c.e.\ sets of the same complexity.
For example, given any c.e.\ set $A$ the set $A\oplus A$ has this property.
In \cite{cek2012} it was shown that every c.e.\ set $A$ is the union of two  
c.e.\ sets $B,C$ such that $B\cap C=\emptyset$
and $A\equiv_K B\equiv_K \equiv C$; also it was shown that this result is not true
in the case of plain initial segment complexity.
The same question has been studied in the context of Turing degrees.
Lachlan \cite{Lachlan:67} showed that there exists a c.e.\ set that cannot be split into
two c.e.\ sets of the same Turing degree. The c.e.\ sets which can be split
into two c.e.\ sets of the same Turing degree are called mitotic and were
studied in \cite{Ladner:73, Ladner:73*1} and \cite{Downey.Slaman:89}.

A related topic of interest concerns the relationship of $\leq_K, \leq_C$ with $\leq_T$
in the context of c.e.\ sets.
This was studied in \cite{MRmerstcdhdtd} more generally, but the arguments used
there provide facts about the c.e.\ case as well. For example, 
the $K$-degree of the halting set does not contain all the c.e.\ members of any
c.e.\ Turing degree. Moreover the same holds for the $C$-degree of the halting set.
Further results 
(along the lines of the theorems in \cite{MRmerstcdhdtd}) 
relating $\leq_K, \leq_C$ with $\leq_T$ in the case of the c.e.\ sets 
 may be obtained by a careful examination of the methods 
in \cite{MRmerstcdhdtd}.

\subsection{Compression with c.e.\ oracles and splittings}\label{subese:compceosplit}
We have already stressed that the $LK$ 
reducibility (measuring the compressing power of oracles) is quite different to $K$
reducibility (measuring the initial segment complexity of reals) even in the context of 
c.e.\ sets. 
However the two measures are quite related on the random sequences. In fact,
it was shown in \cite{milleryutran} that if $X,Y$ are random sequences
then $X\leq_{K} Y$ implies $Y\leq_{LK} X$. 
In other words, a random sequence $X$ can compress finite programs at least
as efficiently as another random sequence $Y$, provided that its initial segments
are at most as complex as those of $Y$.

A similarity between the two measures also occurs in the case of c.e.\ oracles, 
with respect to the 
various splitting properties that where discussed in Section \ref{subse:inicomcesp}.
Such properties of $LK$ have mostly been studied in terms of the related reducibility
$\leq_{LR}$. As we discussed in Section \ref{subse:relcompo} it coincides with 
$\leq_{LK}$. Morever it is rather straightforward to translate an argument concerning $\leq_{LR}$
to the analogous argument about $\leq_{LK}$, by replacing effectively open sets with
prefix-free machines.
The argument in \cite{omunca} may be useful as
a guide for such a translation.
In the following we only use $\leq_{LK}$, although most of the original proofs
in  the literature of the results that we discuss refer to $\leq_{LR}$.
We start with the following analogue of (\ref{eq:splitcl}) that was proved in
\cite{Barmpalias.ea:08} (also see \cite[Footnote 7]{BarmpaliasCompress} and \cite[Chapter 2]{Tom:10}).

\begin{equation}\label{eq:splitLKl}
\parbox{10cm}{If $A$ is c.e.\ set and 
$A >_{LK} \emptyset$ then $A$ is the disjoint union of two 
c.e.\ sets $A_0$, $A_1$ such that 
$A_0 |_{LK} A_1$ and $A_0,A_1 <_{LK} A$.}
\end{equation}
Cone avoidance works for $\leq_{LK}$ much in the same way 
as it does for $\leq_K$.
\begin{equation}\label{eq:splitmoLKl}
\parbox{10cm}{Theorems \ref{th:cafpik} and \ref{thrm_splitk} 
 hold for $\leq_{LK}$ in place of $\leq_K$.}
\end{equation}
The proof of (\ref{eq:splitmoLKl})
does not involve any new ideas, other than the ones presented in 
\cite{Barmpalias.ea:08} and in Section \ref{subse:inicomcesp} of the present article. 
For this reason it is left to the motivated 
reader. Similar cone avoidance arguments have been used in \cite{Morph11}.

The study of c.e.\ sets that cannot be split in the same $LK$ degree was the topic of
\cite[Section 3]{BarmpaliasM09}. 
\begin{equation}\label{eq:spkakaLKl}
\parbox{10cm}{There exists a c.e.\ set $A$ which cannot be split into
two c.e.\ sets $B,C$ such that $A\equiv_{LK} B \equiv_{LK} C$.}
\end{equation}
Moreover the set $A$ that was constructed
was shown to be Turing complete. Recall that the analogue of
(\ref{eq:spkakaLKl}) for $\leq_K$ is an open question.
The following characterization of triviality with respect to $\leq_{LK}$ 
was another result from \cite[Section 3]{BarmpaliasM09}.
\begin{equation}\label{eq:spmorsaLKl}
\parbox{10cm}{A c.e.\ set is $K$-trivial if and only if
it computes a set which cannot be split into two c.e.\ sets of the same
$LK$ degree.}
\end{equation}
The analogue of 
(\ref{eq:spmorsaLKl}) for the case of Turing degrees (i.e.\ that every noncomputable
c.e.\ set computes a non-mitotic set) was shown in \cite{Ladner:73}.
An interesting open question on this topic is the existence
of c.e.\ $LK$ degrees in which all c.e.\ sets can be split into two c.e.\ sets of the same $LK$ degree. 
The analogue of this
question for the Turing degrees was answered positively in
\cite{Ladner:73*1} and was further studied in \cite{Downey.Slaman:89}.


%
%
\end{document}